\documentclass[a4paper,12pt]{article}

\usepackage{color,tikz,amsmath,amssymb,bm}

\unitlength\textwidth
\divide\unitlength by 200\relax

\bmdefine{\sss}{s}
\bmdefine{\vvv}{v}

\DeclareMathAlphabet{\mathscr}{U}{rsfs}{m}{n}

\newcommand{\msPPP}{\mathscr{P}}

\newcommand{\msXXX}{\mathscr{X}}
\newcommand{\msKKK}{\mathscr{K}}

\newcommand{\NNN}{\mathbb{N}}
\newcommand{\ZZZ}{\mathbb{Z}}
\newcommand{\QQQ}{\mathbb{Q}}
\newcommand{\RRR}{\mathbb{R}}
\newcommand{\KKK}{\mathbb{K}}

\newcommand{\pppp}{\mathfrak{p}}

\newcommand{\tUUUUU}{{{t}\mathcal U}}

\newcommand{\cok}{\mathrm{Cok}}
\newcommand{\image}{\mathrm{Im}}
\newcommand{\define}{\mathrel{:=}}

\newcommand{\gor}{Gorenstein}
\newcommand{\cm}{Cohen-Macaulay}

\newcommand{\relint}{{\rm{relint}}}

\newcommand{\trace}{{\mathrm{tr}}}

\newcommand{\Div}{{\mathrm{Div}}}
\renewcommand{\hom}{{\mathrm{Hom}}}
\newcommand{\spec}{{\mathrm{Spec}}}

\newcommand{\aff}{\mathrm{aff}}
\newcommand{\conv}{\mathrm{conv}}
\newcommand{\stab}{\mathrm{STAB}}

\newcommand{\ktmu}{{\KKK[T^{\mu_0}, T^{\mu_1}, \ldots, T^{\mu_{2\ell}}]}}

\newcommand{\xip}{\xi^{+}}

\newcommand{\ekp}{E_\KKK[\msPPP]}

\newcommand{\eksg}{{E_\KKK[\stab(G)]}}

\newtheorem{thm}{Theorem}[section]
\newtheorem{fact}[thm]{Fact}

\newtheorem{lemma}[thm]{Lemma}

\newtheorem{definition}[thm]{Definition}

\newtheorem{conj}[thm]{Conjecture}

\newcommand{\bigzerou}{\smash{\lower1.7ex\hbox{\bg 0}}}

\newcommand{\bigastu}{\smash{\lower1.7ex\hbox{\bg *}}}

\numberwithin{equation}{section}

\newcommand{\mylabel}[1]{{\label{#1}\tt [#1]}}
\let\mylabel=\label

\title{%
Non-\gor\ locus and almost \gor\ property of the Ehrhart ring of
the stable set polytope of a cycle graph%
}
\author{Mitsuhiro Miyazaki%
\footnote{Partially supported by JSPS KAKENHI JP20K03556.}
}
\date{Kyoto University of Education\\
\tt e-mail:mmyzk7@gmail.com}

\begin{document}

\maketitle

\sloppy

\begin{abstract}
%\subjclassname{
Let $R$ be the Ehrhart ring of the stable set polytope of a cycle graph which is not \gor.
We describe the non-\gor\ locus of $\spec R$.
Further, we show that $R$ is almost \gor.
Moreover, we show that the conjecture of Hibi and Tsuchiya is true.
\end{abstract}

\section{Introduction}

In this paper, we call a simple graph consisting of exactly one cycle a cycle graph.
An even cycle graph, i.e., a cycle graph with even vertices is bipartite and therefore is perfect.
By the result of Ohsugi and Hibi \cite[Theorem 2.1 (b)]{oh2}, the Ehrhart ring of the stable set 
polytope of a perfect graph is \gor\ if and only if sizes of maximal cliques are constant.
In particular, the Ehrhart ring of the stable set polytope of an even cycle graph is \gor.

On the other hand, in the course of studying the h-vector of graded \cm\ rings, Hibi and Tsuchiya
\cite[Theorem 1]{ht}
showed that the Ehrhart ring of the stable set polytope of an odd cycle graph is \gor\ if and only
if the size of the cycle is less than or equal to 5.
They used the fact that cycle graphs are t-perfect.
Later the present author vastly generalized this result to general t-perfect graphs and 
characterized when the Ehrhart ring of the stable set polytope is \gor\ completely \cite{miy2}:
the Ehrhart ring of the stable set polytope of a t-perfect graph $G=(V,E)$ is \gor\ if and
only if
(i)
$E=\emptyset$,
(ii)
$G$ has no isolated vertex nor triangle and there is no odd cycle
without chord and length at least 7 or
(iii)
every maximal clique of $G$ has size at least 3
and there is no odd cycle without chord and length at least 5.

In this paper, we study the Ehrhart ring of the stable set polytope of a cycle graph
which is not \gor.
Our main tool of research is the trace ideal of the canonical module.
Herzog, Hibi and Stamate \cite[Lemma 2.1]{hhs} showed that if $R$ is a \cm\ local 
or graded ring over a field with canonical module $\omega_R$, 
then $R_\pppp$ is \gor\ if and only if $\pppp\not\supset\trace(\omega_R)$ for
$\pppp\in\spec R$, where $\trace(\omega_R)$ is the trace of $\omega_R$.
In particular, $\trace(\omega_R)$ is a defining ideal of the non-\gor\ locus
of $R$.
Therefore, one can study the non-\gor\ locus of $\spec R$ of a \cm\ ring $R$
by examining the trace of its canonical module.

In this paper, we first describe the non-\gor\ locus of the $\spec R$
of the Ehrhart ring $R$ of the stable set polytope of an odd cycle graph with length
at least 7.
We describe minimal prime ideals of $\trace(\omega_R)$ explicitly and show that all
of them have exactly the half dimension of the dimension of $R$.

Next we analyze the structure of $\omega_R$ more precisely.
We show that there is a unique monomial in $\omega_R$ with minimal degree.
Moreover, we show that there is a unique system of generators 
of $\omega_R$ consisting of monomials.
By analyzing the margin of monomials in $\omega_R$ with respect to the
odd cycle condition of t-perfect graphs, we classify the monomials in $\omega_R$
and express the structure of $\omega_R$ by using this classification.
Using this expression of the structure of $\omega_R$, we show that $R$
is an almost \gor\ graded ring
%We also study the generating system of $\omega_R$ and show that $R$ is almost \gor\
(for the definition of almost \gor\ property, see \cite{gtt}).

Finally, we study the Ulrich  module appeared in the investigation of almost \gor\
property and show that the conjecture of Hibi and Tsuchiya \cite[Conjecture 1]{ht}
is true.

%%%%%%%%%%%%%%%%%%%%%%%%%%%%%%%%%%%%%%%%%%%

\section{Preliminaries}

\mylabel{sec:pre}

In this section, we establish notation and terminology.
For unexplained terminology of commutative algebra, we consult \cite{bh} and 
of graph theory, we consult \cite{die}.

In this paper, all rings and algebras are assumed to 
be commutative 
with an identity element.
%unless stated otherwise.
Further, all graphs are assumed to be finite, simple and without loop.
We denote the set of nonnegative integers, 
the set of integers, 
the set of rational numbers and 
the set of real numbers
by $\NNN$, $\ZZZ$, $\QQQ$ and $\RRR$ respectively.

For a set $X$, we denote by $\#X$ the cardinality of $X$.
For sets $X$ and $Y$, we define $X\setminus Y\define\{x\in X\mid x\not\in Y\}$.
For nonempty sets $X$ and $Y$, we denote the set of maps from $X$ to $Y$ by $Y^X$.
If $X$ is a finite set, we identify $\RRR^X$ with the Euclidean space 
$\RRR^{\#X}$.
For $f$, $f_1$, $f_2\in\RRR^X$ and $a\in \RRR$,
we define
maps $f_1\pm f_2$ and $af$ 
by
$(f_1\pm f_2)(x)=f_1(x)\pm f_2(x)$ and
$(af)(x)=a(f(x))$
for $x\in X$.
%We denote the zero map $X\ni x\mapsto 0\in \RRR$ by $0$.
%For $f\in\RRR^X$, we set $\supp f\define\{x\in X\mid\xi(x)\neq 0\}$.
Let $A$ be a subset of $X$.
We define the characteristic function $\chi_A\in\RRR^X$ of $A$ by
$\chi_A(x)=1$ for $x\in A$ and $\chi_A(x)=0$ for $x\in X\setminus A$.
For a nonempty subset $\msXXX$ of $\RRR^X$, we denote by $\conv\msXXX$
(resp.\ $\aff\msXXX$)
the convex hull (resp.\ affine span) of $\msXXX$.

\begin{definition}
\mylabel{def:xi+}
\rm
Let $X$ be a finite set and $\xi\in\RRR^X$.
For $B\subset X$, we set $\xip(B)\define\sum_{b\in B}\xi(b)$.
%We define the empty sum to be 0, i.e., if $B=\emptyset$, then $\xi^+(B)=0$.
\end{definition}

A stable set of a graph $G=(V,E)$ is a subset $S$ of $V$ with no two
elements of $S$ are adjacent.
We treat the empty set as a stable set.

\begin{definition}
\rm
The stable set polytope $\stab(G)$ of a graph $G=(V,E)$ is
$$
\conv\{\chi_S\in\RRR^V\mid \mbox{$S$ is a stable set of $G$.}\}
$$
\end{definition}
Note that $\chi_{\{v\}}\in\stab(G)$ for any $v\in V$ and $\chi_\emptyset\in\stab(G)$.
In particular, $\dim\stab(G)=\#V$.

Next we fix notation about Ehrhart rings.
Let $\KKK$ be a field, $X$ a finite set  
and $\msPPP$ a rational convex polytope in $\RRR^X$, i.e., 
a convex polytope whose vertices are contained in $\QQQ^X$.
Let $-\infty$ be a new element
with $-\infty\not\in X$ and
set $X^-\define X\cup\{-\infty\}$.
Also let $\{T_x\}_{x\in X^-}$ be
a family of indeterminates indexed by $X^-$.
For $f\in\ZZZ^{X^-}$, 
we denote the Laurent monomial 
$\prod_{x\in X^-}T_x^{f(x)}$ by $T^f$.
We set $\deg T_x=0$ for $x\in X$ and $\deg T_{-\infty}=1$.
Then the Ehrhart ring of $\msPPP$ over a field $\KKK$ is the $\NNN$-graded subring
$$
\KKK[T^f\mid f\in \ZZZ^{X^-}, f(-\infty)>0, \frac{1}{f(-\infty)}f|_X\in\msPPP]
$$
of the Laurent polynomial ring $\KKK[T_x^{\pm1}\mid x\in X^-]$,
where $f|_X$ is the restriction of $f$ to $X$.
We denote the Ehrhart ring of $\msPPP$ over $\KKK$ by $\ekp$.

It is known that $\ekp$ is Noetherian and $\dim\ekp=\dim\msPPP+1$.
It is also known that $\ekp$ is
normal and \cm\ by the result of Hochster \cite{hoc}.
Moreover, 
by the description of the canonical module of a normal affine semigroup ring
by Stanley \cite[p.\ 82]{sta2}, we see 
the following.

\begin{lemma}
\mylabel{lem:sta desc}
The ideal 
$$\bigoplus_{f\in\ZZZ^{X^-}, f(-\infty)>0, \frac{1}{f(-\infty)}f|_X\in\relint\msPPP}
\KKK T^f
$$
of $\ekp$ is the canonical module of $\ekp$,
where $\relint\msPPP$ denotes the interior of $\msPPP$ in the topological space
$\aff\msPPP$.
\end{lemma}
We denote the ideal of the above lemma by $\omega_{\ekp}$ and call the canonical ideal of $\ekp$.

Let $G=(V,E)$ be a graph.
$G$ is, by definition, t-perfect if
$$
\stab(G)=
\left\{f\in\RRR^V\left|\ 
\vcenter{\hsize=.5\textwidth\relax\noindent
$0\leq f(x)\leq 1$ for any $x\in V$, $f^+(e)\leq 1$ for any $e\in E$ and
$f^+(C)\leq \frac{\#C-1}{2}$ for any odd cycle $C$ in $G$}\right.\right\}.
$$

Let $G=(V,E)$ be a cycle graph with length $n$, i.e., 
$V=\{v_0, v_1, \ldots, v_{n-1}\}$,
$E=\{\{v_i, v_j\}\mid i-j\equiv 1\pmod n\}$.
It is shown by Mahjoub \cite{mah} that $G$ is t-perfect.
%%%%%%%%%%%%%%%%%%%%%%%%%%%%%%%%%%
\iffalse
, i.e.,
$$
\stab(G)=
\left\{f\in\RRR^V\left|\ 
\vcenter{\hsize=.5\textwidth\relax\noindent
$0\leq f(x)\leq 1$ for any $x\in V$, $f^+(e)\leq 1$ for any $e\in E$ and
$f^+(V)\leq \frac{n-1}{2}$ if $n$ is odd}\right.\right\}.
$$
\fi
%%%%%%%%%%%%%%%%%%%%%%%%%%%%%%%
If $n$ is even, then $G$ is a bipartite graph and therefore is perfect.
Thus, by the result of Ohsugi and Hibi
%perfect and by
\cite[Theorem 2.1 (b)]{oh2}, $\eksg$ is \gor.

In this paper, we study $\eksg$ the case where it is not \gor.
Therefore, we assume that the length of $G$ is $2\ell+1$, where $\ell$ is a
positive integer.
Hibi and Tsuchiya \cite[Theorem 1]{ht} (see also \cite[Corollary 3.9]{miy2})
showed that if $G$ is a cycle graph of length $2\ell+1$, then $\eksg$ is \gor\
if and only if $\ell\leq 2$.
Therefore, we mainly consider the case where $\ell \geq 3$.

Our main tool is the result of Herzog, Hibi and Stamate \cite{hhs}
about the trace of the canonical module.

\begin{definition}
\rm
Let $R$ be a ring and $M$ an $R$-module.
We set 
$$
\trace(M)\define\sum_{\varphi\in\hom(M,R)}\varphi(M)
$$
and call $\trace(M)$ the trace of $M$.
\end{definition}
We recall the following.

\begin{fact}[{\cite[Lemma 1.1]{hhs}}]
\mylabel{fact:hhs1.1}
Let $R$ be a ring and $I$ an ideal of $R$ containing an $R$-regular element.
Also let $Q(R)$ be the total quotient ring of fractions of $R$
and set $I^{-1}\define\{x\in Q(R)\mid xI\subset R\}$.
Then
$$
\trace(I)=I^{-1}I.
$$
\end{fact}
Note that if $R$ is a Noetherian normal domain and 
$I$ is a divisorial ideal, then $I^{-1}=I^{(-1)}$, the inverse element of
$I$ in $\Div(R)$.
Moreover, we recall the following.

\begin{fact}[{\cite[Lemma 2.1]{hhs}}]
\mylabel{fact:hhs2.1}
Let $R$ be a \cm\ local or graded ring over a field with canonical
module $\omega_R$.
Then for $\pppp\in\spec(R)$,
$$
R_\pppp\mbox{ is \gor}\iff
\pppp\not\supset\trace(\omega_R).
$$
In particular, the non-\gor\ locus of $\spec R$ is $V(\trace(\omega_R))$.
\end{fact}

We also recall the following our previous results.

\begin{definition}
\rm
\mylabel{def:un}
Let $G'=(V',E')$ be a graph and
set 
$\msKKK=\msKKK(G')\define
\{K\subset V\mid
K$ is a clique of $G'$ and size of $K$ is less than or equal to 3$\}$.
For $n\in\ZZZ$,
we set
$$
%\tUUUUU^{(n)}=
\tUUUUU^{(n)}(G')
\define\left\{\mu\in\ZZZ^{{V'}^-}\left|\ 
\vcenter{\hsize=.5\textwidth\relax\noindent
$\mu(z)\geq n$ for any $z\in V'$,
$\mu^+(K)\leq \mu(-\infty)-n$ for any maximal element $K$ of $\msKKK$ and
$\mu^+(C)\leq\mu(-\infty)\frac{\#C-1}{2}-n$ for any odd cycle $C$ without
chord and length at least $5$}\right.\right\}
$$
\end{definition}
By this notation, the following holds.

\begin{fact}[{\cite[Remark 3.10]{miy2}}]
\mylabel{fact:symb power}
If $G'$ is a t-perfect graph, then
$$
\omega_{E_\KKK[\stab(G')]}^{(n)}=\bigoplus_{\mu\in\tUUUUU^{(n)}(G')}\KKK T^\mu
$$
for any $n\in\ZZZ$,
where $\omega^{(n)}_{E_\KKK[\stab(G')]}$ is the $n$-th power of 
$\omega_{E_\KKK[\stab(G')]}$ in $\Div(E_\KKK[\stab(G')])$.
\end{fact}

We abbreviate $\tUUUUU^{(n)}(G)$ as $\tUUUUU^{(n)}$ in the rest of this paper.
The following lemma is very easily proved but very useful.

\begin{lemma}
\mylabel{lem:triv}
Suppose that $\eta\in\tUUUUU^{(1)}$ 
and $\zeta\in\tUUUUU^{(-1)}$.
If $x\in V$ and $(\eta+\zeta)(x)=0$, then $\eta(x)=1$ and $\zeta(x)=-1$.
\end{lemma}

%%%%%%%%%%%%%%%%%%%%%%%%%%%%%%%%%%%%%%%%%%%%%%

\section{Non-\gor\ loci of the Ehrhart rings of the stable set 
polytopes of cycle graphs}

\mylabel{sec:nongor}

In this section, we state the non-\gor\ loci of the Ehrhart rings of the
stable set polytopes of cycle graphs.
Since it is known that the Ehrhart rings of the stable set polytope of 
even graphs and odd cycle graphs with length at most 5 are \gor,
we focus our attention to odd cycles with length at least 7.

Let $G=(V,E)$ be a cycle graph with length $2\ell+1$, where $\ell$ is an 
integer with $\ell\geq 3$.
We set $R=\eksg$, $V=\{v_0, v_1, \ldots, v_{2\ell}\}$
and $E=\{\{v_i,v_j\}\mid i-j\equiv 1\pmod{2\ell+1}\}$.
Further, we set $e_j=\{v_j, v_{j+1}\}$ for $0\leq j\leq 2\ell-1$
and $e_{2\ell}=\{v_{2\ell}, v_0\}$.
For $i$ with $0\leq i\leq 2\ell$, we define $\mu_i\in\ZZZ^{V^-}$ by
$$
\mu_i(v_j)=
\left\{
\begin{array}{ll}
1&\qquad\mbox{$j-i\equiv 0, 2, 4, \ldots, 2\ell-2 \pmod{2\ell+1}$},\\
0&\qquad\mbox{otherwise}
\end{array}
\right.
$$
and
$$
\mu_i(-\infty)=1
$$
and set $\nu_i=\mu_i|_V$.
%for $0\leq i\leq 2\ell$.
For example, the case where $\ell=4$, $\nu_i$ are as follows, where the
top vertex is $v_0$ and $v_1$, $v_2$, \ldots, $v_8$ are aligned anti-clockwise.
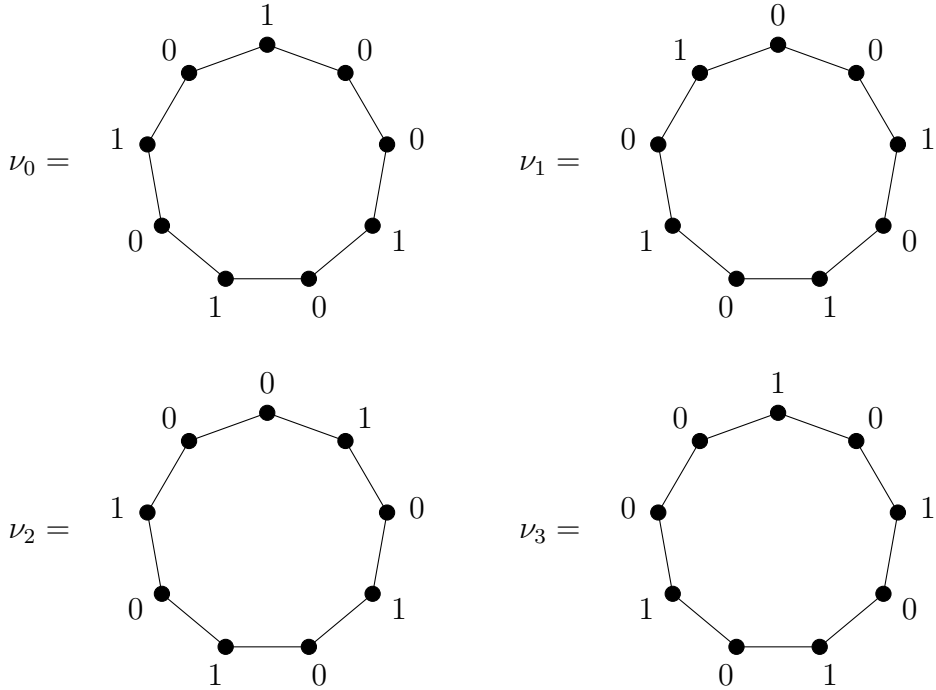
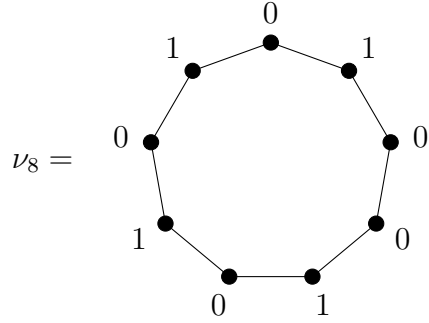
\begin{center}
\begin{tikzpicture}
\node at (180:3){$\nu_0=$};
\coordinate (X0) at (90:1.6);
\coordinate (X1) at (130:1.6);
\coordinate (X2) at (170:1.6);
\coordinate (X3) at (210:1.6);
\coordinate (X4) at (250:1.6);
\coordinate (X5) at (290:1.6);
\coordinate (X6) at (330:1.6);
\coordinate (X7) at (370:1.6);
\coordinate (X8) at (410:1.6);
\draw (X0)--(X1)--(X2)--(X3)--(X4)--(X5)--(X6)--(X7)--(X8)--cycle;
\draw[fill] (X0) circle [radius=0.1];
\draw[fill] (X1) circle [radius=0.1];
\draw[fill] (X2) circle [radius=0.1];
\draw[fill] (X3) circle [radius=0.1];
\draw[fill] (X4) circle [radius=0.1];
\draw[fill] (X5) circle [radius=0.1];
\draw[fill] (X6) circle [radius=0.1];
\draw[fill] (X7) circle [radius=0.1];
\draw[fill] (X8) circle [radius=0.1];

\node at (90:2) {$1$};
\node at (130:2) {$0$};
\node at (170:2) {$1$};
\node at (210:2) {$0$};
\node at (250:2) {$1$};
\node at (290:2) {$0$};
\node at (330:2) {$1$};
\node at (370:2) {$0$};
\node at (410:2) {$0$};
\end{tikzpicture}
\qquad
\begin{tikzpicture}
\node at (180:3){$\nu_1=$};
\coordinate (X0) at (90:1.6);
\coordinate (X1) at (130:1.6);
\coordinate (X2) at (170:1.6);
\coordinate (X3) at (210:1.6);
\coordinate (X4) at (250:1.6);
\coordinate (X5) at (290:1.6);
\coordinate (X6) at (330:1.6);
\coordinate (X7) at (370:1.6);
\coordinate (X8) at (410:1.6);
\draw (X0)--(X1)--(X2)--(X3)--(X4)--(X5)--(X6)--(X7)--(X8)--cycle;
\draw[fill] (X0) circle [radius=0.1];
\draw[fill] (X1) circle [radius=0.1];
\draw[fill] (X2) circle [radius=0.1];
\draw[fill] (X3) circle [radius=0.1];
\draw[fill] (X4) circle [radius=0.1];
\draw[fill] (X5) circle [radius=0.1];
\draw[fill] (X6) circle [radius=0.1];
\draw[fill] (X7) circle [radius=0.1];
\draw[fill] (X8) circle [radius=0.1];

\node at (90:2) {$0$};
\node at (130:2) {$1$};
\node at (170:2) {$0$};
\node at (210:2) {$1$};
\node at (250:2) {$0$};
\node at (290:2) {$1$};
\node at (330:2) {$0$};
\node at (370:2) {$1$};
\node at (410:2) {$0$};
\end{tikzpicture}
\end{center}
\begin{center}
\begin{tikzpicture}
\node at (180:3){$\nu_2=$};
\coordinate (X0) at (90:1.6);
\coordinate (X1) at (130:1.6);
\coordinate (X2) at (170:1.6);
\coordinate (X3) at (210:1.6);
\coordinate (X4) at (250:1.6);
\coordinate (X5) at (290:1.6);
\coordinate (X6) at (330:1.6);
\coordinate (X7) at (370:1.6);
\coordinate (X8) at (410:1.6);
\draw (X0)--(X1)--(X2)--(X3)--(X4)--(X5)--(X6)--(X7)--(X8)--cycle;
\draw[fill] (X0) circle [radius=0.1];
\draw[fill] (X1) circle [radius=0.1];
\draw[fill] (X2) circle [radius=0.1];
\draw[fill] (X3) circle [radius=0.1];
\draw[fill] (X4) circle [radius=0.1];
\draw[fill] (X5) circle [radius=0.1];
\draw[fill] (X6) circle [radius=0.1];
\draw[fill] (X7) circle [radius=0.1];
\draw[fill] (X8) circle [radius=0.1];

\node at (90:2) {$0$};
\node at (130:2) {$0$};
\node at (170:2) {$1$};
\node at (210:2) {$0$};
\node at (250:2) {$1$};
\node at (290:2) {$0$};
\node at (330:2) {$1$};
\node at (370:2) {$0$};
\node at (410:2) {$1$};
\end{tikzpicture}
\qquad
\begin{tikzpicture}
\node at (180:3){$\nu_3=$};
\coordinate (X0) at (90:1.6);
\coordinate (X1) at (130:1.6);
\coordinate (X2) at (170:1.6);
\coordinate (X3) at (210:1.6);
\coordinate (X4) at (250:1.6);
\coordinate (X5) at (290:1.6);
\coordinate (X6) at (330:1.6);
\coordinate (X7) at (370:1.6);
\coordinate (X8) at (410:1.6);
\draw (X0)--(X1)--(X2)--(X3)--(X4)--(X5)--(X6)--(X7)--(X8)--cycle;
\draw[fill] (X0) circle [radius=0.1];
\draw[fill] (X1) circle [radius=0.1];
\draw[fill] (X2) circle [radius=0.1];
\draw[fill] (X3) circle [radius=0.1];
\draw[fill] (X4) circle [radius=0.1];
\draw[fill] (X5) circle [radius=0.1];
\draw[fill] (X6) circle [radius=0.1];
\draw[fill] (X7) circle [radius=0.1];
\draw[fill] (X8) circle [radius=0.1];

\node at (90:2) {$1$};
\node at (130:2) {$0$};
\node at (170:2) {$0$};
\node at (210:2) {$1$};
\node at (250:2) {$0$};
\node at (290:2) {$1$};
\node at (330:2) {$0$};
\node at (370:2) {$1$};
\node at (410:2) {$0$};
\end{tikzpicture}
\end{center}
$$
\cdots
$$
\begin{center}
\begin{tikzpicture}
\node at (180:3){$\nu_8=$};
\coordinate (X0) at (90:1.6);
\coordinate (X1) at (130:1.6);
\coordinate (X2) at (170:1.6);
\coordinate (X3) at (210:1.6);
\coordinate (X4) at (250:1.6);
\coordinate (X5) at (290:1.6);
\coordinate (X6) at (330:1.6);
\coordinate (X7) at (370:1.6);
\coordinate (X8) at (410:1.6);
\draw (X0)--(X1)--(X2)--(X3)--(X4)--(X5)--(X6)--(X7)--(X8)--cycle;
\draw[fill] (X0) circle [radius=0.1];
\draw[fill] (X1) circle [radius=0.1];
\draw[fill] (X2) circle [radius=0.1];
\draw[fill] (X3) circle [radius=0.1];
\draw[fill] (X4) circle [radius=0.1];
\draw[fill] (X5) circle [radius=0.1];
\draw[fill] (X6) circle [radius=0.1];
\draw[fill] (X7) circle [radius=0.1];
\draw[fill] (X8) circle [radius=0.1];

\node at (90:2) {$0$};
\node at (130:2) {$1$};
\node at (170:2) {$0$};
\node at (210:2) {$1$};
\node at (250:2) {$0$};
\node at (290:2) {$1$};
\node at (330:2) {$0$};
\node at (370:2) {$0$};
\node at (410:2) {$1$};
\end{tikzpicture}
\end{center}

Let $\pppp_i$ be the ideal of $R$ generated by $\{T^\mu\mid
\mu\in\tUUUUU^{(0)}$, $\mu(v_i)>0$ or $\mu^+(V)<\ell\mu(-\infty)\}$,
i.e., 
$$\pppp_i=\bigoplus_{\mu\in\tUUUUU^{(0)}\atop\mu(v_i)>0\ \mathrm{or}\ \mu^+(V)<\ell\mu(-\infty)}
\KKK T^\mu$$
for $0\leq i\leq 2\ell$.
Then $\pppp_i$ is the prime ideal corresponding to the face
$\msPPP_i=\{f\in\stab(G)\mid f(v_i)=0, f^+(V)=\ell\}$ of $\stab(G)$,
i.e., $E_\KKK[\msPPP_i]=R/\pppp_i$.

We first analyze the dimension of $\msPPP_i$.
By symmetry, we see that $\dim \msPPP_i=\dim\msPPP_0$ for any $i$.

If $f\in\msPPP_0$, then $f(v_0)=0$ and $f^+(V)=\ell$.
Since
$f^+(V)=f(v_0)+\sum_{j=1}^\ell f^+(e_{2j-1})$ and
$f^+(e_j)\leq 1$ for any $j$,
we see that
$$
f^+(e_{2j-1})=1\qquad\mbox{for $1\leq j\leq \ell$}.
$$
These linear equations and $f(v_0)=0$  are independent.
Therefore we see that $\dim\msPPP_0\leq 2\ell+1-(\ell+1)=\ell$.

Next we prove  the reverse inequality.
It is easily verified that $\nu_2$, $\nu_4$, \ldots, $\nu_{2\ell}$ and $\nu_1$ are
elements of $\msPPP_0$.
The matrix whose columns correspond to $v_0$, $v_1$, \ldots, $v_{2\ell}$ and rows correspond
to $\nu_4-\nu_2$, $\nu_6-\nu_4$, \ldots, $\nu_{2\ell}-\nu_{2\ell-2}$, $\nu_1-\nu_{2\ell}$ 
respectively is
$$
\left[
\begin{array}{ccccccccccc}
0&1&-1&0&0&\cdots&\cdots&0&0&0&0\\
0&0&0&1&-1&\cdots&\cdots&0&0&0&0\\
&&&&\ddots&\ddots\\
&&&&&\ddots&\ddots\\
&&&&&&\ddots&\ddots\\
0&0&0&0&0&\cdots&\cdots&1&-1&0&0\\
0&0&0&0&0&\cdots&\cdots&0&0&1&-1
\end{array}
\right].
$$
This is a rank $\ell$ matrix.
Therefore, $\dim\msPPP_0=\ell$.
In particular, $\dim R/\pppp_0=\ell+1$.

Next we state the following.

\begin{lemma}
\mylabel{lem:not in pi}
Let $i$ be an integer with $0\leq i\leq 2\ell$.
Then $\pppp_i\supset\trace(\omega_R)$.
\end{lemma}
\begin{proof}
We may assume that $i=0$.
Let $T^\mu$ be an arbitrary monomial in $\trace(\omega_R)$.
We deduce a contradiction by assuming $T^\mu\not\in\pppp_0$.

Since $T^\mu\in\trace(\omega_R)$, there are $\eta\in\tUUUUU^{(1)}$ and $\zeta\in\tUUUUU^{(-1)}$
with $\mu=\eta+\zeta$.
Since $T^\mu\not\in\pppp_0$, it holds that $\mu(v_0)=0$ and therefore
$\eta(v_0)=1$ by Lemma \ref{lem:triv}.
Thus,
$$
\eta^+(V)=\sum_{j=1}^\ell \eta^+(e_{2j-1})+1.
$$
Since $\eta^+(e_{2j-1})+1\leq\eta(-\infty)$ for $1\leq j\leq \ell$, we see that
$$
\eta^+(V)\leq\sum_{j=1}^\ell(\eta(-\infty)-1)+1=\ell\eta(-\infty)-\ell+1.
$$
On the other hand, since $T^\mu\not\in\pppp_0$,
$$
\mu^+(V)=\ell\mu(-\infty).
$$
Moreover, since
\begin{eqnarray*}
&&\eta^+(V)+1\leq\ell\eta(-\infty)\\
&&\zeta^+(V)-1\leq\ell\zeta(-\infty)\\
&&\eta(-\infty)+\zeta(-\infty)=\mu(-\infty)
\end{eqnarray*}
and
$$
\eta^+(V)+\zeta^+(V)=\mu^+(V)=\ell\mu(-\infty),
$$
we see that
$$
\eta^+(V)+1=\ell\eta(-\infty).
$$
Therefore,
$$
\ell\eta(-\infty)-1=\eta^+(V)\leq\ell\eta(-\infty)-\ell+1
$$
and we see that $\ell\leq 2$.
This contradicts to the assumption.
\end{proof}

By this lemma, we see that 
$$
\trace(\omega_R)\subset\bigcap_{i=0}^{2\ell}\pppp_i.
$$
Since the right hand side is a radical ideal, we see that
$$
\sqrt{\trace(\omega_R)}\subset\bigcap_{i=0}^{2\ell}\pppp_i.
$$

In order to show the reverse inclusion, we first state the following.

\begin{lemma}
\mylabel{lem:all positive}
If $\mu\in\tUUUUU^{(0)}$ and $\mu(v_i)>0$ for any $i$, then
$T^\mu\in\sqrt{\trace(\omega_R)}$.
\end{lemma}
\begin{proof}
Define $\eta$, $\zeta\in\ZZZ^{V^-}$ by
$$
\eta(x)=
\left\{
\begin{array}{ll}
\ell-1,&\qquad\mbox{$x\in V$},\\
2\ell-1,&\qquad\mbox{$x=-\infty$}
\end{array}
\right.
$$
and
$$
\zeta(x)=
\left\{
\begin{array}{ll}
(\ell-2)\mu(x)-\ell+1,&\qquad\mbox{$x\in V$},\\
(\ell-2)\mu(x)-2\ell+1,&\qquad\mbox{$x=-\infty$}.
\end{array}
\right.
$$
Then $\eta+\zeta=(\ell-2)\mu$.

For $x\in V$, 
$$\eta(x)=\ell-1\geq 1$$ and
$$\zeta(x)=(\ell-2)\mu(x)-\ell+1\geq \ell-2-\ell+1=-1,$$
since $\ell\geq 3$ and $\mu(x)\geq 1$.
Further, for $e\in E$, 
$$\eta^+(e)+1=2(\ell-1)+1=2\ell-1=\eta(-\infty)$$ and
$$\zeta^+(e)-1=(\ell-2)\mu^+(e)-2(\ell-1)-1
\leq(\ell-2)\mu(-\infty)-2\ell+1=\zeta(-\infty).$$
Finally,
\begin{eqnarray*}
\eta^+(V)+1&=&(2\ell+1)(\ell-1)+1\\
&=&2\ell^2-\ell\\
&=&\ell(2\ell-1)\\
&=&\ell\eta(-\infty)
\end{eqnarray*}
and
\begin{eqnarray*}
\zeta^+(V)-1&=&(\ell-2)\mu^+(V)-(2\ell+1)(\ell-1)-1\\
&\leq& \ell(\ell-2)\mu(-\infty)-2\ell^2+\ell\\
&=&\ell((\ell-2)\mu(-\infty)-2\ell+1)=\ell\zeta(-\infty).
\end{eqnarray*}
Therefore, $\eta\in\tUUUUU^{(1)}$ and $\zeta\in\tUUUUU^{(-1)}$.
Thus, we see that
$$
(T^\mu)^{\ell-2}=T^{(\ell-2)\mu}=T^\eta T^\zeta
\in\omega_R\omega_R^{(-1)}=\trace(\omega_R)
$$
and therefore
$$
T^\mu\in\sqrt{\trace(\omega_R)}.
$$
\end{proof}

Next we state the following.

\begin{lemma}
\mylabel{lem:less than lmu}
Let $\mu\in\tUUUUU^{(0)}$ and $\mu^+(V)<\ell\mu(-\infty)$.
Then $T^\mu\in\sqrt{\trace(\omega_R)}$.
\end{lemma}
\begin{proof}
Define $\eta$, $\zeta\in\ZZZ^{V^-}$ by
$$
\eta(x)=
\left\{
\begin{array}{ll}
1,&\qquad\mbox{$x\in V$},\\
3,&\qquad\mbox{$x=-\infty$}
\end{array}
\right.
$$
and
$$
\zeta(x)=
\left\{
\begin{array}{ll}
(\ell-2)\mu(x)-1,&\qquad\mbox{$x\in V$},\\
(\ell-2)\mu(-\infty)-3,&\qquad\mbox{$x=-\infty$}.
\end{array}
\right.
$$
Then $\eta+\zeta=(\ell-2)\mu$.

It is obvious that $\eta(x)\geq 1$ and $\zeta(x)\geq -1$ for any $x\in V$.
Let $e$ be an arbitrary edge in $G$.
Then
$$
\eta^+(e)+1=2+1=\eta(-\infty)
$$
and
$$
\zeta^+(e)-1=(\ell-2)\mu^+(e)-2-1\leq(\ell-2)\mu(-\infty)-3=\zeta(-\infty).
$$
%for any $e\in E$.
Further,
$$
\eta^+(V)+1=(2\ell+1)+1\leq3\ell=\ell\eta(-\infty)
$$
and, since $\mu^+(V)+1\leq\ell\mu(-\infty)$ by assumption,
we see that
\begin{eqnarray*}
\zeta^+(V)-1&=&(\ell-2)\mu^+(V)-(2\ell+1)-1\\
&=&(\ell-2)(\mu^+(V)+1)-3\ell\\
&\leq&(\ell-2)\ell\mu(-\infty)-3\ell\\
&=&\ell\zeta(-\infty).
\end{eqnarray*}
Therefore, $\eta\in\tUUUUU^{(1)}$ and $\zeta\in\tUUUUU^{(-1)}$.
Thus, we see that
$$
(T^\mu)^{\ell-2}=T^{(\ell-2)\mu}=T^\eta T^\zeta
\in\omega_R\omega_R^{(-1)}=\trace(\omega_R)
$$
and therefore
$$
T^\mu\in\sqrt{\trace(\omega_R)}.
$$
\end{proof}

By Lemmas \ref{lem:not in pi}, \ref{lem:all positive} and \ref{lem:less than lmu},
we see the following.

\begin{thm}
\mylabel{thm:non gor}
Let $G=(V,E)$ be a cycle graph with length $2\ell+1$, where $\ell$
is an integer with $\ell\geq 3$.
Then in the above notation,
$$
\sqrt{\trace(\omega_R)}=\bigcap_{i=0}^{2\ell}\pppp_i.
$$
In particular, non-\gor\ locus of the Ehrhart ring $R$ of the stable
set polytope of $G$ is a closed subset of $\spec R$ of
dimension $\ell+1$.
\end{thm}

%%%%%%%%%%%%%%%%%%%%%%%%%%%%%%%%%%%%%%%%%%%%%%

\section{Almost \gor\ property}

\mylabel{sec:ag}

In this section, we show that the Ehrhart rings of the stable set polytopes
of cycle graphs are almost \gor\ graded rings.
For the definition of almost \gor\ property, see \cite{gtt}.
In this paper, we only treat almost \gor\ graded property and we say almost
\gor\ graded as almost \gor\ for short.
We focus our attention to odd cycle graphs of length at least 7 by the
same reason as the previous section.

Let $\ell$, $G$, $R$, $v_i$, $e_i$, $\mu_i$ and $\nu_i$ for $0\leq i\leq 2\ell$ be
as in the previous section.
Further, for integer $k$ with $1\leq k\leq \ell-1$, we define $\eta_k\in\ZZZ^{V^-}$
by
$$
\eta_k(x)=
\left\{
\begin{array}{ll}
k,&\qquad\mbox{$x\in V$},\\
2k+1,&\qquad\mbox{$x=-\infty$}.
\end{array}
\right.
$$
Then it is easily verified that $\eta_k\in\tUUUUU^{(1)}$ for any $k$.
%%%%%%%%%%%%%%%%%%%%%%
\iffalse
For $\eta\in\tUUUUU^{(1)}$, we define the margin of $\eta$ by
$$
\ell\eta(-\infty)-\eta^+(V).
$$
Then $\eta_k$ is an element of $\tUUUUU^{(1)}$ with margin $\ell-k$
for any $k$ with $1\leq k\leq \ell-1$.
\fi
%%%%%%%%%%%%%%%%%%%%%
Note that $\ell\eta_k(-\infty)-\eta_k^+(V)=\ell-k$ and $\eta_k^+(e)+1=\eta_k(-\infty)$ for any
$e\in E$ and $1\leq k\leq \ell-1$.

Since 
for any $\eta\in\tUUUUU^{(1)}$ and for any $i$ with $0\leq i\leq 2\ell$,
it holds that
$$
\eta(-\infty)\geq\eta^+(e_i)+1\geq 3,
$$
we see that $T^{\eta_1}$ is the unique monomial in $\omega_R$ with minimum degree 3.
In particular, $a(R)=-\deg T^{\eta_1}=-3$.
Therefore, we consider the morphism $\varphi\colon R\to\omega_R(3)$ of graded
$R$-modules with $\varphi(1)=T^{\eta_1}$.
$R$ is, by definition, almost \gor\ if and only if $\cok\varphi$ is an Ulrich module.

\begin{lemma}
\mylabel{lem:im phi}
It holds that
$$
(\image\varphi)(-3)=\bigoplus_{\eta\in\tUUUUU^{(1)}\atop\ell\eta(-\infty)-\eta^+(V)\geq\ell-1}
\KKK T^\eta.
$$
\end{lemma}
\begin{proof}
First note that $(\image\varphi)(-3)=T^{\eta_1}R$.

Let $\eta$ be an arbitrary element of $\tUUUUU^{(1)}$ with 
$\ell\eta(-\infty)-\eta^+(V)\geq\ell-1$.
Set $\mu=\eta-\eta_1$.
Then 
$$
\mu(x)=\eta(x)-\eta_1(x)=\eta(x)-1\geq 0
$$
for any $x\in V$,
\begin{eqnarray*}
\mu^+(e)&=&\eta^+(e)-\eta_1^+(e)\\
&=&\eta^+(e)-2\\
&=&(\eta^+(e)+1)-3\\
&\leq&\eta(-\infty)-\eta_1(-\infty)\\
&=&\mu(-\infty)
\end{eqnarray*}
for any $e\in E$ and
\begin{eqnarray*}
\ell\mu(-\infty)-\mu^+(V)&=&\ell\eta(-\infty)-\eta^+(V)-(\ell\eta_1(-\infty)-\eta_1^+(V))\\
&\geq&(\ell-1)-(\ell-1)\\
&=&0.
\end{eqnarray*}
Therefore, we see that $\mu\in\tUUUUU^{(0)}$.
Thus, we see that
$$
T^\eta=T^{\eta_1}T^\mu\in T^{\eta_1}R.
$$

On the other hand, if $T^\eta$ is a monomial in $T^{\eta_1}R$, then there is
$\mu\in\tUUUUU^{(0)}$ with $\eta=\mu+\eta_1$.
Since $T^\eta\in T^{\eta_1}R\subset\omega_R$, we see that
$\eta\in\tUUUUU^{(1)}$ by Fact \ref{fact:symb power}.
Further,
\begin{eqnarray*}
\ell\eta(-\infty)-\eta^+(V)&=&
\ell\mu(-\infty)-\mu^+(V)+\ell\eta_1(-\infty)-\eta_1^+(V)\\
&\geq&\ell\eta_1(-\infty)-\eta_1^+(V)\\
&=&\ell-1.
\end{eqnarray*}
Thus we see that
$$
(\image\varphi)(-3)=T^{\eta_1}R
=\bigoplus_{\eta\in\tUUUUU^{(1)}\atop\ell\eta(-\infty)-\eta^+(V)\geq\ell-1}\KKK T^\eta.
$$
\end{proof}

 Next, we set
$$
\msPPP\define\{f\in\stab(G)\mid f^+(V)=\ell\}.
$$
Then $\msPPP$ is a face of $\stab(G)$.
Further, we set 
$$
\tUUUUU^{(0)}_0\define\{\mu\in\tUUUUU^{(0)}\mid\mu^+(V)=\ell\mu(-\infty)\}
$$
and
$$
R^{(0)}\define\bigoplus_{\mu\in\tUUUUU^{(0)}_0}\KKK T^\mu.
$$
Then $R^{(0)}$ is a subalgebra of $R$ and the Ehrhart ring $E_\KKK[\msPPP]$ of $\msPPP$.
Note that $\mu_i\in\tUUUUU^{(0)}_0$ for $0\leq i\leq 2\ell$.

\begin{lemma}
\mylabel{lem:r0kt}
It holds that
$$
%R^{(0)}=\KKK[T^{\mu_0}, T^{\mu_1}, \ldots, T^{\mu_{2\ell}}].
R^{(0)}=\ktmu.
$$
\end{lemma}
\begin{proof}
Since
$\mu_i\in\tUUUUU^{(0)}_0$ for $0\leq i\leq 2\ell$, it is clear that 
%$R^{(0)}\supset\KKK[T^{\mu_0}, T^{\mu_1}, \ldots, T^{\mu_{2\ell}}]$.
$R^{(0)}\supset\ktmu$.
Let $\mu$ be an arbitrary element of $\tUUUUU^{(0)}_0$.
We prove by induction on $\mu(-\infty)$ that $T^\mu\in\ktmu$.

If $\mu(-\infty)=0$, then $\mu=0$ and 
$T^\mu=1\in\ktmu$.
Suppose that $\mu(-\infty)>0$.
We first consider the case where $\mu(x)>0$ for any $x\in V$.
Since 
$$
(2\ell+1)\mu(-\infty)>2\ell\mu(-\infty)=2\mu^+(V)=\sum_{i=0}^{2\ell}\mu^+(e_i),
$$
we see that there is $i$ with $\mu^+(e_i)<\mu(-\infty)$.
By symmetry, we may assume that 
$$
\mu^+(e_0)<\mu(-\infty).
$$
Set $\mu'=\mu-\mu_2$.
Then
$$
\mu'(x)=\mu(x)-\mu_2(x)\geq \mu(x)-1\geq0
$$
for any $x\in V$, since $\mu(x)>0$.
If $e\in E$ and $e\neq e_0$, then 
$$
(\mu')^+(e)=\mu^+(e)-\mu_2^+(e)=\mu^+(e)-1\leq\mu(-\infty)-1=\mu'(-\infty).
$$
Further, 
$$
(\mu')^+(e_0)=\mu^+(e_0)-\mu^+_2(e_0)=\mu^+(e_0)<\mu(-\infty)
$$
by assumption.
Therefore,
$$
(\mu')^+(e_0)\leq\mu(-\infty)-1=\mu'(-\infty).
$$
Moreover,
$$
(\mu')^+(V)=\mu^+(V)-\mu_2^+(V)=\ell\mu(-\infty)-\ell\mu_2(-\infty)=\ell\mu'(-\infty).
$$
Therefore, $\mu'\in\tUUUUU^{(0)}_0$ and by induction hypothesis, we see that
$
T^{\mu'}\in\ktmu
$.
Thus,
$$
T^\mu=T^{\mu'}T^{\mu_2}\in\ktmu.
$$

Next consider the case where $\mu(x)=0$ for some $x\in V$.
By symmetry, we may assume that $\mu(v_0)=0$.
Since
$
\sum_{i=0}^\ell\mu^+(e_{2i-1})=
\mu(v_0)+\sum_{i=0}^\ell\mu^+(e_{2i-1})=\mu^+(V)=\ell\mu(-\infty)
$
%, $\mu(v_0)=0$ 
and 
$\mu^+(e_j)\leq\mu(-\infty)$ for any $j$,
we see that
$$
\mu(e_{2i-1})=\mu(-\infty)\qquad\mbox{for any $1\leq i\leq \ell$}.
$$

First consider the case where $\mu(v_{2i})=0$ for any $1\leq i\leq \ell$.
In this case, 
$\mu(v_{2i-1})=\mu^+(e_{2i-1})=\mu(-\infty)>0$ for any $1\leq i\leq \ell$.
Set $\mu'=\mu-\mu_1$.
Then
$$
\mu'(v_{2i-1})=\mu(v_{2i-1})-\mu_1(v_{2i-1})=\mu(v_{2i-1})-1\geq 0
$$
for $1\leq i\leq \ell$ and
$$
\mu'(x)=\mu(x)-\mu_1(x)=\mu(x)\geq 0
$$
for any $x\in V\setminus\{v_1, v_3, \ldots, v_{2\ell-1}\}$.
Further,
$$
(\mu')^+(e_{2i-1})=\mu^+(e_{2i-1})-\mu^+_1(e_{2i-1})
=\mu(-\infty)-1=\mu'(-\infty)
$$
for any $1\leq i\leq \ell$,
$$
(\mu')^+(e_{2i})=\mu^+(e_{2i})-\mu^+_1(e_{2i})=\mu(v_{2i+1})-1=\mu(-\infty)-1=\mu'(-\infty)
$$
for $0\leq i\leq \ell-1$ and
$$
(\mu')^+(e_{2\ell})=\mu^+(e_{2\ell})-\mu_1^+(e_{2\ell})=\mu(v_{2\ell})+\mu(v_0)=0\leq
\mu(-\infty)-1=\mu'(-\infty).
$$
Moreover,
$$
(\mu')^+(V)=\mu^+(V)-\mu_1^+(V)=\ell\mu(-\infty)-\ell\mu_1(-\infty)=\ell\mu'(-\infty).
$$
Therefore, $\mu'\in\tUUUUU^{(0)}_0$ and by induction hypothesis, we see that
$T^{\mu'}\in\ktmu$.
Thus,
$$
T^\mu=T^{\mu'}T^{\mu_1}\in\ktmu.
$$

If $\mu(v_{2i})\neq0$ for some $i$ with $1\leq i\leq\ell$, set
$j=\min\{i\mid 1\leq i\leq \ell, \mu(v_{2i})\neq 0\}$.
For $i$ with $1\leq i\leq j-1$,
$$
\mu(v_{2i-1})=\mu(v_{2i-1})+\mu(v_{2i})=\mu^+(e_{2i-1})=\mu(-\infty)>0,
$$
since $\mu^+(e_{2i-1})=\mu(-\infty)$ and $\mu(v_{2i})=0$.
Further, if $j<\ell$, then $\mu(v_{2j})>0$ and 
$\mu(v_{2j})+\mu(v_{2j+1})=\mu^+(e_{2j})\leq\mu(-\infty)$,
we see that $\mu(v_{2j+1})<\mu(-\infty)$.
Moreover, since
$\mu(v_{2j+1})+\mu(v_{2j+2})=\mu^+(e_{2j+1})=\mu(-\infty)$,
we see that $\mu(v_{2j+2})>0$.
By the same argument and induction, we see that
$$
\mu(v_{2i})>0 \qquad \mbox{for $j\leq i\leq \ell$}.
$$
Set $\mu'=\mu-\mu_{2j}$.
Since
$$
\mu_{2j}(v_k)=
\left\{
\begin{array}{ll}
1,&\qquad\mbox{$k\in\{1,3,\ldots, 2j-3, 2j,2j+2, \ldots, 2\ell\}$},\\
0,&\qquad\mbox{otherwise}
\end{array}
\right.
$$
and
$$
\mu(v_{k})>0\qquad\mbox{if $k\in\{1,3,\ldots, 2j-3, 2j,2j+2, \ldots, 2\ell\}$},
$$
we see that
$$
\mu(x)\geq 0\qquad\mbox{for any $x\in V$}.
$$
Further, if $i\neq 2j-2$, then,
since $\mu^+_{2j}(e_i)=\mu_{2j}(-\infty)=1$,
we see that
$$
(\mu')^+(e_i)=\mu^+(e_i)-\mu^+_{2j}(e_i)
%=\mu^+(e_i)-1
\leq\mu(-\infty)-\mu_{2j}(-\infty)=\mu'(-\infty).
$$
%since $\mu^+_{2j}(e_i)=1$.
Moreover, since
$\mu(v_{2j-1})<\mu(v_{2j-1})+\mu(v_{2j})=\mu^+(e_{2j-1})
%=\mu^+(e_{2j-1})-\mu(v_{2j})<\mu^+(e_{2j-1})
=\mu(-\infty)$ and
$\mu(v_{2j-2})=0$ 
by the definition of $j$,
%and $e_{2j-2}=\{v_{2j-2},v_{2j-1}\}$
we see that
$$
(\mu')^+(e_{2j-2})=\mu^+(e_{2j-2})-\mu^+_{2j}(e_{2j-2})
=\mu(v_{2j-2})+\mu(v_{2j-1})\leq\mu(-\infty)-1=\mu'(-\infty).
$$
Finally,
$$
(\mu')^+(V)=\mu^+(V)-\mu^+_{2j}(V)=\ell\mu(-\infty)-\ell\mu_{2j}(-\infty)=\ell\mu'(-\infty).
$$
Thus, we see that $\mu'\in\tUUUUU^{(0)}_0$.
Therefore, by induction hypothesis, we see that
$T^{\mu'}\in\ktmu$
and
$$
T^\mu=T^{\mu'}T^{\mu_{2j}}\in\ktmu.
$$
\end{proof}

Next we consider the dimension of $\msPPP$.
The matrix whose columns correspond to $v_0$, $v_1$ \ldots, $v_{2\ell}$ and rows
correspond to $\nu_3-\nu_1$, $\nu_4-\nu_2$, \ldots, $\nu_{2\ell}-\nu_{2\ell-2}$, 
$\nu_0-\nu_{2\ell-1}$ and $\nu_1-\nu_{2\ell}$ respectively is
$$
\left[
\begin{array}{cccccccccc}
1&-1&0&0&\cdots&\cdots&0&0&0\\
0&1&-1&0&\cdots&\cdots&0&0&0\\
0&0&1&-1&\cdots&\cdots&0&0&0\\
\vdots&\vdots&&\ddots&\ddots&&\vdots&\vdots&\vdots\\
\vdots&\vdots&&&\ddots&\ddots&\vdots&\vdots&\vdots\\
\vdots&\vdots&&&&\ddots&\ddots&\vdots&\vdots\\
0&0&0&0&\cdots&\cdots&1&-1&0\\
0&0&0&0&\cdots&\cdots&0&1&-1\\
\end{array}
\right].
$$
This is a matrix of rank $2\ell$.
Since $\nu_i\in\msPPP$ for $0\leq i\leq 2\ell$,
we see that $\dim\msPPP\geq 2\ell$.
On the other hand, $E_\KKK[\msPPP]=R^{(0)}=\ktmu$ by Lemma \ref{lem:r0kt}.
Since $\dim E_\KKK[\msPPP]=\dim\msPPP+1$, we see that 
$\dim\msPPP=2\ell$.
Moreover, we see that $T^{\mu_0}$, $T^{\mu_1}$, \ldots, $T^{\mu_{2\ell}}$
are algebraically independent over $\KKK$.
Since $\deg T^{\mu_i}=\mu(-\infty)=1$ for $0\leq i\leq 2\ell$, we see the following.

\begin{lemma}
\mylabel{lem:isom poly}
$R^{(0)}$ is isomorphic to the polynomial ring with $2\ell+1$ variables equipped with
the standard grading.
\end{lemma}

For $k$ with $2\leq k\leq\ell-1$, we set
$$
C_k\define\bigoplus_{\eta\in\tUUUUU^{(1)}\atop\ell\eta(-\infty)-\eta^+(V)=\ell-k}\KKK T^\eta.
$$
Then we see the following.

\begin{lemma}
\mylabel{lem:ck r0 free}
$C_k$ is a rank 1 free $R^{(0)}$-module with basis $T^{\eta_k}$
for $2\leq k\leq \ell-1$.
\end{lemma}
\begin{proof}
First, if $\mu\in\tUUUUU^{(0)}_0$, then it is easily verified that
$\mu+\eta_k\in\tUUUUU^{(1)}$ and
$$
\ell(\mu+\eta_k)(-\infty)-(\mu+\eta_k)^+(V)
=\ell\eta_k(-\infty)-\eta^+_k(V)=\ell-k.
$$
Therefore $T^\mu T^{\eta_k}\in C_k$.

Conversely, assume that $\eta\in\tUUUUU^{(1)}$ and $\ell\eta(-\infty)-\eta^+(V)=\ell-k$.
Set $\mu=\eta-\eta_k$.
Then 
\begin{eqnarray*}
\mu^+(V)&=&\eta^+(V)-\eta^+_k(V)\\
&=&\ell\eta(-\infty)-(\ell-k)-(\ell\eta_k(-\infty)-(\ell-k))\\
&=&\ell(\eta(-\infty)-\eta_k(-\infty))\\
&=&\ell\mu(-\infty)
\end{eqnarray*}
and, since $\eta_k^+(e)+1=\eta_k(-\infty)$,
\begin{eqnarray*}
\mu^+(e)&=&(\eta^+(e)+1)-(\eta^+_k(e)+1)\\
%&\leq&\eta(-\infty)-(2k+1)\\
&\leq&\eta(-\infty)-\eta_k(-\infty)\\
&=&\mu(-\infty).
\end{eqnarray*}

Finally, we show that $\mu(x)\geq 0$ for any $x\in V$.
Assume the contrary.
Then, by symmetry, we may assume that $\mu(v_0)<0$.
Then
$$
\mu^+(V)=\mu(v_0)+\sum_{i=1}^\ell\mu^+(e_{2i-1})<\ell\mu(-\infty),
$$
contradicting the fact shown above.
Therefore, $\mu(x)\geq 0$ for any $x\in V$ and we see that $\mu\in \tUUUUU^{(0)}_0$.
Thus,
$$
C_k=\left(\bigoplus_{\mu\in\tUUUUU^{(0)}_0}\KKK T^\mu\right)T^{\eta_k}=T^{\eta_k}R^{(0)}.
$$
Since $C_k$ and $R^{(0)}$ are contained in a domain $R$, we see that 
$C_k$ is a rank 1 free $R^{(0)}$-module with basis $T^{\eta_k}$.
\end{proof}

Since 
\begin{eqnarray*}
\omega_R&=&\bigoplus_{\eta\in\tUUUUU{(1)}}\KKK T^\eta\\
&=&
\left(\bigoplus_{\eta\in\tUUUUU^{(1)}\atop\ell\eta(-\infty)-\eta^+(V)\geq\ell-1}\KKK T^\mu\right)
\oplus\left(\bigoplus_{k=2}^{\ell-1}\left(\bigoplus_{\eta\in\tUUUUU^{(1)}\atop\ell\eta(-\infty)-\eta^+(V)=\ell-k}\KKK T^\eta\right)\right)\\
&=&(\image\varphi)(-3)\oplus\left(\bigoplus_{k=2}^{\ell-1} C_k\right),
\end{eqnarray*}
we see that
$$
\cok\varphi\cong\left(\bigoplus_{k=2}^{\ell-1} C_k\right)(3)
$$
as graded $R^{(0)}$-modules.
Since each $C_k$ is a free $R^{(0)}$-module, 
$R^{(0)}$ is isomorphic to a polynomial ring with $2\ell+1$ variables over $\KKK$
and multiplicity of a module can be computed by its Hilbert series,
we see that
$$
e(\cok\varphi)=\ell-2.
$$

Next, we show that $T^{\eta_1}$, $T^{\eta_2}$, \ldots, $T^{\eta_{\ell-1}}$ is a minimal
system of generators of $\omega_R$.
Assume the contrary.
Then there are $i$ and $j$ with $i\neq j$ and $\mu\in\tUUUUU^{(0)}$ such that
$$
T^{\eta_i}=T^\mu T^{\eta_j}.
$$
Since $i=\eta_i(v_0)=\mu(v_0)+\eta_j(v_0)\geq j$,
we see that $i>j$ and 
$\mu(x)=\eta_i(x)-\eta_j(x)=i-j$ for any $x\in V$.
On the other hand, it holds that
$$
\mu(-\infty)=\eta_i(-\infty)-\eta_j(-\infty)=(2i+1)-(2j+1)=2(i-j).
$$
Therefore,
$$
\ell\mu(-\infty)-\mu^+(V)=2\ell(i-j)-(2\ell+1)(i-j)<0,
$$
contradicting the fact that $\mu\in\tUUUUU^{(0)}$.

Thus, we see that $T^{\eta_1}$, $T^{\eta_2}$, \ldots, $T^{\eta_{\ell-1}}$ is a minimal
system of generators of $\omega_R$ and
$$
\mu(\cok\varphi)=\ell-2.
$$
Therefore, we see the following.

\begin{thm}
\mylabel{thm:almost gor}
$\cok\varphi$ is an Ulrich module and $R=\eksg$ is an almost \gor\ ring.
\end{thm}

%%%%%%%%%%%%%%%%%%%%%%%%%%%%%%%%%%%%%%%%%%%%%%

\section{Hibi-Tsuchiya's conjecture}
\mylabel{sec:ht conj}

Hibi and Tsuchiya \cite[Conjecture 1]{ht} made a conjecture about the
h-vector of the Ehrhart rings of cycle graphs.
In this section, we prove that the conjecture is true.

Let $(h_0, h_1, \ldots, h_s)$, $h_s\neq0$ be the h-vector of $R=\eksg$.
Since $\dim R=2\ell+2$ and $a(R)=-3$, we see that $s=\dim R+a(R)=2\ell-1$.
Hibi and Tsuchiya made the following.

\begin{conj}
\mylabel{conj:ht}
If $\ell\geq 3$, then the h-vector of $R$ is the following form.
$$
(1, h_1, h_2, \ldots, h_{\ell-1}, h_{\ell-1}+(-1)^{\ell-1}, h_{\ell-2}+(-1)^{\ell-2},\ldots, 
h_3-1, h_2+1, h_1, 1).
$$
\end{conj}

Now we prove the following.

\begin{thm}
\mylabel{thm:ht conj}
Conjecture \ref{conj:ht} is true.
\end{thm}
\begin{proof}
We use the notation of the previous section.
Also for a finitely generated graded $R$-module $M$,
we denote by $H(M,\lambda)$ the Hilbert series of $M$, i.e.,
$$
H(M,\lambda)=\sum_{n\in\ZZZ}(\dim_\KKK M_n)\lambda^n.
$$
Then 
$$
H(R,\lambda)=\frac{h_0+h_1\lambda+\cdots+h_s\lambda^s}{(1-\lambda)^{2\ell+2}}.
$$
Further, by the second proof of \cite[Theorem 4.1]{sta2}, we see that
$$
H(\omega_R(3),\lambda)=\frac{h_s+h_{s-1}\lambda+\cdots+h_0\lambda^s}{(1-\lambda)^{2\ell+2}}.
$$
On the other hand, since $C_k$ is a rank 1 free $R^{(0)}$-module with basis
$T^{\eta_k}$ and $\deg T^{\eta_k}=2k+1$, we see that
$$
H(C_k,\lambda)=\frac{\lambda^{2k+1}}{(1-\lambda)^{2\ell+1}}
$$
as $R^{(0)}$-modules
for $2\leq k\leq \ell-1$.
Further, since $\cok\varphi=\left(\bigoplus_{k=2}^{\ell-1}C_k\right)(3)$,
we see that
$$
\dim_\KKK(\cok\varphi)_n=\sum_{k=2}^{\ell-1}\dim_\KKK(C_k)_{n+3}
$$
for any $n\in\ZZZ$. Thus
\begin{eqnarray*}
H(\cok\varphi,\lambda)&=
&\frac{\lambda^2+\lambda^4+\cdots\lambda^{2\ell-4}}{(1-\lambda)^{2\ell+1}}\\
&=&\frac{\lambda^2-\lambda^3+\lambda^4-\lambda^5+\cdots+\lambda^{2\ell-4}-\lambda^{2\ell-3}}
{(1-\lambda)^{2\ell+2}}.
\end{eqnarray*}

Since 
$$
H(\omega_R(3),\lambda)=H(R,\lambda)+H(\cok\varphi,\lambda),
$$
we see that 
$$
h_s=h_0, \quad h_{s-1}=h_1 \quad\mbox{ and }\quad
h_{s-i}=h_i+(-1)^i\quad \mbox{for $2\leq i\leq 2\ell-3$}.
$$
The assertion follows from these equations, since $h_0=1$.
\end{proof}

%%%%%%%%%%%%%%%%%%%%%%%%%%%%%%%%%%%%%%%%%%%%%%

%===========================

%

\end{document}